\documentclass[preprint,11pt]{paper}
\usepackage[T1]{fontenc}
\usepackage[latin1]{inputenc}
\usepackage{ucs}
\usepackage[english]{babel}
\usepackage{multicol}
\usepackage{amssymb}
\usepackage{amsmath}
\usepackage{float}
\usepackage{theorem}
\usepackage{graphicx, epsfig}
\usepackage{verbatim}
\usepackage{bbm}

\usepackage[official,right]{eurosym}

\makeatletter
\DeclareRobustCommand{\qed}{%
  \ifmmode 
  \else \leavevmode\unskip\penalty9999 \hbox{}\nobreak\hfill
  \fi
  \quad\hbox{\qedsymbol}}
\newcommand{\openbox}{\leavevmode
  \hbox to.77778em{%
  \hfil\vrule
  \vbox to.675em{\hrule width.6em\vfil\hrule}%
  \vrule\hfil}}
\newcommand{\qedsymbol}{\openbox}
\newenvironment{proof}[1][Proof]{\par
  \normalfont
  \topsep6\p@\@plus6\p@ \trivlist
  \item[\hskip\labelsep\bfseries
    #1\@addpunct{.}]\ignorespaces
}{%
  \qed\endtrivlist
}
\makeatother

\newtheorem{defi}{Definition}
\newtheorem{prop}{Proposition}
\newtheorem{lem}{Lemma}

{\theorembodyfont{\rmfamily} \newtheorem{rem}{Remark}}

\newcommand{\id}[1]{\ensuremath{\mathbbm{1}_{#1}}}
\newcommand{\bc}{\begin{center}}
\newcommand{\ec}{\end{center}}
\newcommand{\1}{\mathbf{1}}

\newcommand{\E}{\mathbb{E}}

\renewcommand{\P}{\mathbb{P}}

\newcommand{\Q}{\mathbb{Q}}
\newcommand{\R}{\mathbb{R}}

\newcommand{\cC}{\mathcal{C}}

\newcommand{\cF}{\mathcal{F}}

\newcommand{\Wsd}{\mathbb{W}^{SD}}

\newcommand{\pare}[1]{\left ( #1 \right )}

\def \defi{:=}

\begin{document}

\begin{frontmatter}



\title{Exact simulation of one-dimensional stochastic differential equations involving the local time at zero of the unknown process}


\author[pe]{Pierre \'Etor\'e\corref{cor1}\fnref{fn1}}
\ead{pierre.etore@imag.fr }
\author[mm]{Miguel Martinez}
\ead{miguel.martinez@univ-mlv.fr}

\address[pe]{Laboratoire Jean Kuntzmann- Tour IRMA 51, rue des Math\'ematiques, 38041 Grenoble Cedex 9, France. }
\address[mm]{Universit\'e Paris-Est, Laboratoire d'Analyse et de Math\'ematiques Appliqu\'ees, UMR $8050$, 5  Bld Descartes, Champs-sur-marne, 77454 Marne-la-Vall\'ee Cedex 2, France. }

\cortext[cor1]{Corresponding author}
\fntext[fn1]{ Phone: + 33 (0)4 76 51 45 57; Fax: +33 (0)4 76 63 12 63 }

\begin{abstract}
In this article we extend the exact simulation methods of Beskos et al. in \cite{beskos2} to the solutions of one-dimensional stochastic differential equations involving the local time of the unknown process at point zero. In order to perform the method we compute the law of the skew Brownian motion with drift.  The method presented in this article covers the case where the solution of the SDE with local time corresponds to a divergence form operator with a discontinuous coefficient at zero. Numerical examples are shown to illustrate the method and the performances are compared with more traditional discretization schemes.
\end{abstract}

\begin{keyword}

Exact simulation methods ; Skew Brownian motion ; One-dimensional diffusion ; Local Time.
\end{keyword}

\end{frontmatter}

\section{Introduction}
		\subsection{Presentation}
		
		Exact simulation methods for trajectories of one-dimensional SDEs has been a subject of much interest in the last years : see for example \cite{beskos1}, \cite{beskos2}, \cite{beskos5}, \cite{tanre1}, \cite{sbai}. Unlike the classical simulation methods
which all involve some kind of discretization error (we mention \cite{ball:tala:96:1} for the Euler Scheme), the exact simulation methods are constructed in such a way that they do not present any discretization error under the strong hypothesis that the diffusion coefficient is constant and equal to one. In the last years, the original method presented in the
fundamental article  \cite{beskos2} has been extended to overcome various limitations of the initial algorithm ; it has been 
generalized to include the cases of unbounded drifts (\cite{beskos5}, \cite{beskos4}).

On another hand, the numerical simulation of SDEs corresponding to divergence form operators
$$L=\nabla.\pare{\frac{1}{2}a(x)\nabla}$$ involving a discontinuous coefficient $a$ has been also the subject of various studies in the last years. Indeed these operators are of great importance since they appear in a wide range of modelling problems involving diffusion phenomena in discontinuous media. Among applications, we  can mention ecology 
(\cite{cantrel-cosner}), geophysics (\cite{lejay-geo1}, \cite{lejay-geo2}),
astrophysics (\cite{zhang}), or magneto/electroencephalography  (\cite{faugeras}).

In the one-dimensional context, various Random Walks and an Euler Scheme have been studied for the simulation of the solution of such SDEs  : for Random Walks we mention \cite{etore05a}, \cite{etore06a}, \cite{etore-lejay06a}, \cite{lejay-martinez04a}~; for the Euler Scheme see \cite{martinez04a}, \cite{martinez06a}, \cite{martinez12a} in the case where the discontinuity of the coefficient in the divergence operator appears at point $0$. Of course, for such SDEs, the order of discretization error of these discretization schemes is usually greater than those obtained in a more classical context. 

An important problem comes from the fact that SDEs corresponding to divergence form operators do not enter the classical scope of SDEs covered by the exact simulation methods. 

The main difficulty is that these SDEs include an additional term, which involves in dimension one the local time of the unknown process (in dimension greater than one, it involves the local time of a  one-dimensional auxiliary process; see \cite{maire-champagnat-10}). In fact, the laws of the solution of such one-dimensional SDEs are no longer absolutely continuous with respect to the Wiener measure. 

In this paper we present a first attempt for the adaptation of the exact simulation methods of \cite{beskos2} to one-dimensional SDEs with an additional term that involves the local time of the unknown process at point $0$. Namely, our object of study is $(X_t)_{t\geq 0}$ solution of 
\begin{equation}
\label{SDE-intro}
dX_t=\sigma(X_t)dW_t+\bar{b}(X_t)dt+\beta dL^0_t(X),\hspace{1,0 cm}X_0=x,
\end{equation}
where $|\beta|<1$ (with $\beta\neq 0$) and $L^0_t(X)$ is the symmetric local time of $X$ in zero at time $t$~;~the reason why we only deal with $|\beta|<1$ is that there is no solution to (\ref{SDE-intro}) when $\beta >1$ (the case where $|\beta|=1$ corresponds to a reflected diffusion  and we do not include it in our discussion). The reason why we restrict ourselves to the case $\beta\neq 0$ is made clear in Section 3 (see Remark \ref{rem-betaneqzero} and also the conclusion of the paper). 

Note that when $\sigma$ is identically equal to $1$ and $\bar{b}$ is identically equal to $0$, the solution $(X_t)$ of (\ref{SDE-intro}) is a standard Skew Brownian Motion (SBM in short). 
Under mild assumptions concerning $\bar{b}$ and standard ellipticity conditions on $\sigma$, it is known that
there exists a unique strong solution $(X_t)_{t\geq 0}$ to  (\ref{SDE-intro}) (see \cite{legall} for details).

Let us emphasize that this work includes the situation where $\bar{b}$ may be discontinuous at $0$. So that the results of this paper are also suited for the situation stated in \cite{martinez06a}, \cite{martinez12a} where the solution of (\ref{SDE-intro}) corresponds to a divergence form operator whose coefficient is discontinuous at $0$ (and is sufficiently smooth elsewhere).  We show a numerical example to illustrate this interesting case.

Let us now briefly explain our main idea.
When $\sigma\equiv 1$, we show that the law of $(X_t)_{t\geq 0}$ (solution of (\ref{SDE-intro})) is absolutely continuous with respect to the law of some Skew Brownian Motion (SBM) with a drift component. The reason why the SBM with drift appears naturally in our computations is explained in Section \ref{section-exact-sim}
(see Remark \ref{justification}).

So, contrary to the already mentioned discretization schemes where the standard SBM is used in force, we do not longer deal with a simple SBM but with a SBM that possesses a drift component. As a consequence, in order to adapt the method of \cite{beskos2} in this setting, we have to be able to simulate bridges of the SBM with drift.

In the last section of the paper, we discuss the limitations of the initial algorithm.

The main issue is to relax the boundedness assumptions made on the drift function $\bar{b}$, as is done in \cite{beskos5} for "classical" SDEs. In \cite{beskos5}, the authors use some kind of factorisations for the sample state space of the standard Brownian Bridge, which are consequences of William's decomposition theorem for Brownian Motion. Proving similar factorisations for the Skew Brownian Bridge with drift seems difficult to us. Nevertheless, we have been able to apply a result stated in Pitman-Yor \cite{Pitman-Yor} in the case of the standard Skew Brownian Bridge, which gives a first partial result. Unfortunately, we have not been able to relax the boundedness assumption on the drift function $\bar{b}$ and we think that much remains to do in this direction.

In our concluding remarks we discuss the particular problem of being able to produce an exact simulation algorithm in the case $\beta=0$, which in our opinion should be regarded as a special separate problem. We also draw bold lines for further investigation. 
		\subsection{Organisation of the paper}
		
		In Section \ref{sec-pb} we precise the hypotheses and define the problem we will deal with. We also introduce notations used in the sequel. In Section \ref{section-exact-sim} we present the exact simulation algorithm, adapted from Beskos et al. to our situation. It turns out that, in order to use the algorithm, we need to sample bridges of a Skew Brownian Motion with drift. Section \ref{sec-sbmdrift} is devoted to the computation of the transition probability density of the Skew Brownian Motion with drift. Then Section \ref{ss:section:bridge} explains how to sample bridges of a Skew Brownian Motion with drift, using rejection sampling with brownian bridges as proposals. 
		Section \ref{sec-num} presents numerical experiments, including a divergence form case. Finally,  we discuss possible extensions in Section \ref{sec-conclu}.

\section{Exposition of the exact simulation problem.}
\label{sec-pb}
		\subsection{Exact simulation problem and first assumptions}
		
Denote $C=C([0,T],\R)$ the set of continuous mappings from $[0,T]$ to $\R$ and $\cC$ the Borel 
$\sigma$-field on $C$ induced by the supreme norm. 

Let $\P$ be a probability measure on $(C,\cC)$ and $W$ a Brownian motion under $\P$ together with its completed natural filtration $\pare{{\cal F}_t}_{t\geq 0}$. We will denote $\P^{x}=\P\pare{\cdot\,|\,W_0=x}$.

Throughout the whole paper, we will make the following assumptions
\begin{list}{--}{}
\item $|\beta | <1$. 
\item 
\label{assumption-1}
The function $\bar{b}:\R\to\R$ is bounded and differentiable on ${\mathbb R}^{\ast,+}$ and ${\mathbb R}^{\ast,-}$ with a possible discontinuity at point $\{0\}$.
We suppose that both limits $\lim_{z\rightarrow 0+}\bar{b}(z):=\bar{b}(0+)$ and $\lim_{z\rightarrow 0-}\bar{b}(z):=\bar{b}(0-)$ exist and are finite. The value $\bar{b}(0)$ of the function $\bar{b}$ at $0$ is
of no importance and can be fixed arbitrarily to some constant (possibly different from either $\bar{b}(0+)$ or $\bar{b}(0-)$).

\label{assumption-2} 
\end{list}

We seek for an exact simulation algorithm of the paths of the solution of the one-dimensional Stochastic Differential Equation
\begin{equation}
\label{edstl}
dX_t=dW_t+\bar{b}(X_t)dt+\beta dL^0_t(X),\hspace{1,0 cm}X_0=x,
\end{equation}
where $L^0_t(X)$ is the symmetric local time of $X$ in zero at time $t$.

		\subsection{Some recalls on SDEs of type \eqref{edstl}}
		
		\subsubsection{Existence and uniqueness}
Under the assumptions of the previous section, the equation (\ref{edstl}) possesses a unique strong solution. In fact, performing the bijective change of variable $g(x)\mapsto (1-\beta)x\id{x\geq 0} + (1+\beta)x\id{x<0}$ allows to consider $Y_t:=g(X_t)$ solution of a new transformed equation 
$$dY_t = 1/(g^{-1}_{\pm})'(Y_s)dW_s + \bar{b}\circ g^{-1}/(g^{-1}_{\pm})'(Y_s)ds$$ 
without local time. Here $(g^{-1}_{\pm})'$ denotes the half sum of the right and left derivatives of $g^{-1}$.  

Since $(g^{-1}_{\pm})'$ is bounded from below by a strictly positive constant, this transformed equation makes sense and well-known results for classical one dimensional SDEs (see for example \cite{kara} Chap 5. Section 5.5) ensure that under the assumptions of Subsection \ref{assumption-1}, it possesses a weak solution. Then, $g^{-1}(Y)$ gives a weak solution of \eqref{edstl}. 

The difficulty concerns strong uniqueness. Strong uniqueness for the solutions of equation (\ref{edstl}) is proved with a direct application of Theorem 1.3 p.55 in the fundamental article \cite{legall}, which deals with a broader class of stochastic differential equations involving the local time of the unknown process.

Note that when $\beta=-1$ or $\beta=+1$, equation (\ref{edstl}) possesses a unique strong solution, which is a reflected diffusion at $0$, either reflected below $0$ ($\beta = -1$) or above $0$ ($\beta= +1$). 

Let us now briefly explain why there is no solution to equation (\ref{edstl}) when $|\beta|>1$. Remember that we are working with the symmetric local time $L_t^0(X)$. Let us denote by $L_t^{0,r}(X)$ (resp. $L_t^{0,l}(X)$) the right-hand sided local time of the process $X$ (resp. the left hand sided local time of $X$). It is an exercise to prove that if $X$ is a solution of (\ref{edstl}), then $L_t^0(X) = \frac{1+\beta}{2}L_t^{0,r}(X)$ and $L_t^0(X) = \frac{1-\beta}{2}L_t^{0,l}(X)$. In particular, we see that when $|\beta|>1$ there is no solution to (\ref{edstl}) (otherwise the symmetric local time of $X$ would be negative !). 

\subsubsection{Strong Markov property}
\label{sssection-markov}

The proof of the strong Markov property for solutions of equation (\ref{edstl}) is a separate problem from the one of existence and uniqueness. We refer to \cite{Kulik} for a rigorous proof. In the multidimensional context of diffusion processes with generalized drift, the proof of the Markov property may be found in \cite{Zaitseva-1}.

\section{Exact simulation}
\label{section-exact-sim}
		\subsection{Notations and additional assumptions}
		
\subsubsection{Notations}
\label{notations}
Throughout the whole paper, we use the following notations~:~

\begin{list}{--}{}
\item We note
$$N^c(x):=\frac{1}{\sqrt{2\pi}}\int_x^\infty e^{-\frac{y^2}{2}}dy.$$

\item 
We set  
\begin{equation}
\label{mu-def}
\mu:=\frac{1+\beta}{2\beta}\bar{b}(0+) - \frac{1-\beta}{2\beta}\bar{b}(0-).
\end{equation}
and define $b(z):=\bar{b}(z)-\mu$. 

We set 
$$z\mapsto \phi(z):=\displaystyle \frac{b^2(z)+b'(z) + 2\mu b(z)}{2}\id{{\mathbb R}^{\ast,+}\cup {\mathbb R}^{\ast,-}}(z).$$

\item We set $\tilde{\phi}(z):=\phi(z)-m$ with $m=\inf_{z\in\R}\phi(z)$; the constant $K$ denotes an upper bound of the function $\tilde{\phi}$.

\item  We set $B(u):=\displaystyle \int_0^u b(y)dy$, $u\in {\mathbb R}$. 

\item $B^{\beta,\mu}$ will denote the SBM of parameter $\beta$ and drift $\mu$. That is to say $B^{\beta,\mu}$ is the strong solution of
\eqref{edstl} in the case $\bar{b}\equiv \mu$, namely~:~
\begin{equation}
\label{edsmu}
dB^{\beta,\mu}_t=dW_t+\mu dt+\beta dL^0_t(B^{\beta,\mu}).
\end{equation}
We will denote $p^{\beta,\mu}(t,x,y)$ the transition probability density of $B^{\beta,\mu}$. 

Note that, with this notation,  $p^{0,\mu}(t,x,y)$ is the transition probability density
of the Brownian motion with constant drift $\mu\in\R$, namely
\begin{equation}
\label{def-pmu}
p^{0,\mu}(t,x,y)=\frac{1}{\sqrt{2\pi t}}\exp\big\{  -\frac{(y-x-\mu t)^2}{2t}  \big\}.
\end{equation}
Note also that $p^{\beta,0}(t,x,y)$ is the transition probability density of the SBM of parameter $\beta$ (without drift), see \cite{lejay-2006}.

\item We set 
$$
\tau_0:=\inf\pare{t\geq 0~:~B^{\beta,\mu}_t=0}\,\,\;\text{with the convention}\;\;\;\inf(\emptyset)=+\infty.
$$
We will denote by $h(x,.)$ the density of $\tau_0$ under $\P^x$.

\end{list}

\subsubsection{Additional assumptions}
\label{assumptions}

In this section, we will make the following additional assumptions~:~
\begin{list}{--}{}

\item $\beta\neq 0$ (for the reason of this last hypothesis, see Remark \ref{rem-betaneqzero}).

\item We assume that the function
$z\mapsto \phi(z)$ is bounded.

\item We assume that the function $u\mapsto\exp[B(u)-(u-x)^2/2T]$ is integrable.

\end{list}				
		
	    \subsection{Presentation of the exact simulation algorithm}
	     \label{sec:algo_nondege}
	    
\subsubsection{Change of probability}
		Recall that in this case $b(z):=\bar{b}(z)-\mu$ where $\mu$ is the constant defined by $$\displaystyle \mu:=\frac{1+\beta}{2\beta}\bar{b}(0+) - \frac{1-\beta}{2\beta}\bar{b}(0-).$$ Note that since $\beta \neq 0$ by assumption, this constant is well-defined. In the case where $\bar{b}$ is continuous at point $\{0\}$, observe that
$\mu$ reduces to $\bar{b}(0)$.

We have
$$dX_t=dW_t+b(X_t)dt+\mu dt+\beta dL^0_t(X).$$


In particular, we may perform Girsanov's theorem (see Theorem 3.5.1 in \cite{kara}) and we write
\begin{equation}
\label{sbmssQSD}
dX_t=dW^{SD}_t+\mu dt+\beta dL^0_t(X),
\end{equation}
where $\displaystyle W^{SD}_t\defi W_t+\int_0^tb(X_s)ds$ is a Brownian motion under the new probability $\Wsd$ defined by
\begin{equation}
\label{girsanov-bar}
\dfrac{d\P}{d\Wsd}=\exp\Big\{\int_0^Tb(X_t)dW^{SD}_t-\frac{1}{2}\int_0^Tb^2(X_t)dt \Big\}.
\end{equation}
From our assumptions on $b$, we are in position to apply the symmetric It\^o-Tanaka formula to the function $B(u)\defi \displaystyle\int_0^ub(y)dy$ and $(X_t)_{t\geq 0}$. 

Applying the occupation's time formula, we obtain
\begin{equation}
\label{eq:tpslocdisp}
\begin{split}
B(X_T)-B(x)=&\int_0^T\frac{b(X_t+)+b(X_t-)}{2}dX_t+\frac{1}{2}\int_0^Tb'(X_t)\id{X_t\neq 0}dt+\frac{b(0+)-b(0-)}{2}L_T^0(X)\\
=&\int_0^Tb(X_t)\id{X_t\neq 0}dW^{SD}_t  + \mu\int_0^Tb(X_t)\id{X_t\neq 0}dt\\
&+\frac{1}{2}\int_0^Tb'(X_t)\id{X_t\neq 0}dt+\underbrace{\pare{\frac{b(0+)+b(0-)}{2}\beta + \frac{{b}(0+)-{b}(0-)}{2}}}_{=\,0}L_T^0(X)
\end{split}
\end{equation}
where the last line comes from the definitions of $b$ and $\mu$ (and the property $dL_t^0(X)=\1_{X_t= 0}dL_t^0(X)$). From the fact that $\ell\{t\in [0,T]:X_t=0\}=0$ (where $\ell$ stands for the Lebesgue measure), we see that
\begin{equation}
\label{eq-justification-drift}
B(X_T)-B(x)=\int_0^Tb(X_t)dW^{SD}_t+\mu\int_0^Tb(X_t)dt+\frac{1}{2}\int_0^Tb'(X_t)dt.
\end{equation}
Thus, (\ref{girsanov-bar}) implies that for any functional $F(X)$ of the path up to time $T$, one has~:~
$$
\E_{\P}[F(X)]=\E_{\Wsd}\big[F(X)\exp\big\{B(X_T)-B(x)-\int_0^T\phi(X_t)dt \big\}  \big],
$$
 where $\phi(z)=\dfrac{b^2(z)+b'(z)+2\mu b(z)}{2}$.
 
 \begin{rem}
 \label{justification}
 Note that, because of the definition of $b$, there is no local time appearing in equality (\ref{eq-justification-drift}) after the application of the It\^o-Tanaka formula. This ensures that there
 is no local time involved in the exponential martingale of Girsanov's theorem, which makes it tractable for a numerical perspective. 
 
Retrospectively, this explains why in the sequel we have to deal with a Skew Brownian Motion with drift instead of a simple standard SBM. 
 \end{rem}
 
  \begin{rem}
 \label{rem-betaneqzero}
 We now explain our assumption $\beta\neq 0$. 
 
 Note that in the case $\beta=0$, the constant $\mu$ is no more defined. In fact, in the case $\beta=0$, because of the discontinuity of $\bar{b}$, it is no longer possible to get rid of the local time as in \eqref{eq:tpslocdisp}. More precisely, there is no constant $\theta$ such that proceeding as the computations in \eqref{sbmssQSD} and \eqref{girsanov-bar} with $b(x):=\bar{b}(x)-\theta$  we can cancel the local time term appearing in the exponential weight.

For a more detailed discussion on the case $\beta= 0$ the interested reader is invited to read the conclusion at the end of this paper. 
   \end{rem}
 
		\subsubsection{Exact simulation algorithm (after Beskos and al)}
		\label{subsubsec:exact-sim-algo}

Considering 
\eqref{sbmssQSD}, we see that the law of $X$ under $\Wsd$ is given by 
$p^{\beta,\mu}(t,x,y)dy$. 
\vspace{0,3 cm}

\vspace{0.2cm}
Following the lines of Beskos et al. in \cite{beskos2}, and considering the computations performed in the above section, we give an algorithm that returns an
exact simulation of a skeleton of $(X_t)_{t\in [0,T]}$ solution of (\ref{edstl}) starting from $x_0$.

\vspace{0.2cm}
\begin{center}
\hrulefill

\textsc{  
EXACT SIMULATION ALGORITHM FOR A SOLUTION OF (\ref{edstl}) starting from $x_0$.
\begin{enumerate}
\item Simulate a random variable $Z$ according to the density 
$$
h(y)=C\exp\pare{B(y)-B(x_0)} p^{\beta,\mu}(T,x_0,y),
$$
where $C$ is the normalizing constant such that $\int h(y)dy = 1$. 
Keep in memory the value $z$ of $Z$.
\item Simulate a Poisson Point Process with unit density on $[0,T]\times [0,K]$. The result is a random number $n$ of points of coordinates $(t_1,z_1),\dots, (t_n,z_n)$.
\item Simulate a skeleton $(B^{\beta,\mu}_{t_1},\ldots,B^{\beta,\mu}_{t_n})$ conditioned on $B^{\beta,\mu}_0=x_0$ and $B^{\beta,\mu}_T=z$.
\item If $\forall i\in \{1,\dots, n\}$ $\tilde{\phi}({B}^{\beta, \mu}_{t_i})\leq z_i$ accept the skeleton. Else return to step 1.
\end{enumerate}  
}

\hrulefill
\end{center}
\vspace{0.2cm}

This algorithm returns an exact sample of $(X_{t_1},\ldots,X_{t_n},X_T)$ (in particular we get an exact simulation of $X_T$, it is the value $z$ of $Z$ used for an accepted trajectory). 

\vspace{0.1cm} 
Note that in order to apply the methodology of \cite{beskos2} we have to be able to generate bridges of a drifted Skew Brownian Motion 
$B^{\beta,\mu}$. Indeed, this is the key one has to reach for in order to perform the Step 3.
	    
	    \section{Computation of the law of the SBM with drift}
	    \label{sec-sbmdrift}
\subsection{Recalls on the construction of the SBM using a "random flipping" of excursions}
\label{subsec:random-flipping}
In this paragraph, we present a construction of the Skew Brownian Motion - solution of (\ref{edsmu}) with $\mu=0$ and starting from $x=0$ - that gives an understanding of its relation with the standard Brownian motion. The construction is made out from a reflecting Brownian Motion with a change of sign of each excursion with probability $\pare{1-\beta}/{2}$. It is explained in \cite{RY} page 487 exercise 2.16 (we use the same notations as \cite{RY} in the explanations below). 

Suppose that $x=0$. The construction is as follows~:~let $(Y_n)_{n\geq 0}$ be a sequence of independent r.v.'s taking the values $1$ and $-1$ with probabilities $\pare{1+\beta}/{2}$ and $\pare{1-\beta}/{2}$ and independent of some Brownian Motion $B$. Let us set $\pare{{\cal F}_t^B}_{t\geq 0}$ the natural filtration generated by $B$ and satisfying the usual right continuous and completeness conditions and ${\cal F}^B:=\bigvee_{t\geq 0}{\cal F}^B_t$.

We also denote ${\cal H}:= \sigma(Y_n : n\geq 0)$ the corresponding $\sigma$-algebra generated by the whole sequence $(Y_n)_{n\geq 0}$ and ${\cal E}:=\sigma({\rm e}_s : s\geq 0)$ the $\sigma$-algebra generated by all the excursions of $B$ so that ${\cal E}\subset {\cal F}^B$. The "good" time clock for the point-wise excursion process $({\rm e}_{s})_{s\geq 0}$ is the process of time-change $(\tau_t)_{t\geq 0}$ defined as the r.c.l inverse of the local time $\pare{L_t^0\pare{B}}_{t\geq 0}$~:~so that the excursion process $({\rm e}_{s})_{s\geq 0}$ may be viewed as a Point Poisson Process on ${\cal C}_{0\rightarrow 0}$ running in the local time scale.  

For each $\omega$ in the set on which $B$ is defined, the set of excursions ${\rm e}_s(\omega)$ is countable and may be ordered. Define a process $B^\beta$ by putting $B^\beta_t(\omega) = Y_{n_s(\rm \bf e)}(\omega)|{\rm e}_s(t-\tau_{s-}(\omega), \omega)|$ if 
$\tau_{s-}\leq t\leq \tau_{s}$ and where ${\rm e}_s$ is the $n_s({\rm \bf e})$-th excursion in the above ordering. The random number $n_s({\rm \bf e})$ is a random variable measurable w.r.t. ${\cal E}$, which depends on the whole excursion process ${\rm \bf e} = ({\rm e}_u)_{u>0}$ and the time variable $s$ in the local time scale. It may be proved that the process thus obtained is a Markov process and that it is a Skew Brownian motion of parameter $\beta$ starting at $x=0$.


By construction the sigma algebras ${\cal H}= \sigma(Y_n~:~n\geq 0)$ and ${\cal F}^B$ are independent~; in particular, if we denote by $R(\rm e)$ the end point of excursion $\rm e$, then we have $\tau_t(\omega) = \sum\limits_{s\leq t}R({\rm e}_s(\omega))$ and thus $\tau_t$ is measurable w.r.t. ${\cal F}^B$. Note that this construction implies that almost surely, $L^0_t(B^\beta)=L^0_t(|B|) = L^0_t(B)$ for any  $t\geq 0$. So $\pare{L_t(B^\beta)}_{t\geq 0}$ may be recovered as the r.c.l inverse of $(\tau_t)_{t\geq 0}$. Consequently, it is adapted to $\pare{{\cal F}^B_t}_{t\geq 0}$ and it is independent of ${\cal H}$.

\begin{rem}
For a possible extension of this "random flipping of excursions" method for the construction of the solution to the more general equation (\ref{edstl}) (and possibly solutions of (\ref{SDE-intro})), we refer to the article of Lejay \cite{lejay-decomp}, which gives a decomposition of the It\^o measure associated to $X$ in the general context of solutions of (\ref{SDE-intro}). At least in the context of equation (\ref{edstl}), the result stated in \cite{lejay-decomp} should allow to perform a construction along the same lines as above, flipping the excursions of some reflected process whose law should be the same as the law of $|X|$. However and up to our knowledge, such construction has never been explicitly written down in the literature, even in the context of equation (\ref{edstl}).
\end{rem}

\subsection{Computation of the joint law of SBM and its local time}
Let us begin with a direct consequence of the construction explained above in Subsection \ref{subsec:random-flipping}.
\begin{lem}
\label{lemsko}
Let $W$ be a Brownian motion defined on $(C,\cC,\P)$ and $B^{\beta}$ the strong solution of \eqref{edsmu} with $\mu=0$. 

We have for all $t> 0$, 
\begin{equation}
\label{BLP0}
\P^0\big[\, |B^\beta_t| \in dy ; L_t^0(B^\beta)\in dl \, \big]
=\P^0\big[\, |W_t| \in dy ; L_t^0(W)\in dl \, \big].
\end{equation}
\end{lem}

\begin{rem}
We even have that the  process $(|B^\beta_t|, L_t^0(B^\beta))_{t\geq 0}$ is distributed as
$(|W_t|, L_t^0(W))_{t\geq 0}$ under $\P^0$. This is a clear consequence of the construction explained in paragraph \ref{subsec:random-flipping}. Another way to prove this fact is to check that their common distribution is the one of
$(M^W_t-W_t,M^W_t )_{t\geq 0}$ under $\P^0$, where $M^W_t=\max_{0\leq s\leq t}W_s$. This is related to the L\'evy theorem, as stated for instance in Theorem 3.6.17 in \cite{kara}, where it is proved by using the Skorokhod method. 

\end{rem}

Let us now state an intuitive result, which is somewhat not so easy to prove without using the construction explained in Subsection \ref{subsec:random-flipping}. The difficulty comes from the presence of the local time in the equalities below.

\begin{lem}
\label{excursion}
We have for all $t> 0$,
$$
\P^0\big[\, B^\beta_t \in dy ; L_t^0(B^\beta)\in dl    \, \big]=\frac{1+\beta}{2} \P^0\big[\, |B^\beta_t| \in dy ; L_t^0(B^\beta)\in dl \big]
+\frac{1-\beta}{2} \P^0\big[\, -|B^\beta_t| \in dy ; L_t^0(B^\beta)\in dl \big].
$$
\end{lem}
\begin{proof}

We start from the construction of SBM using a random flipping of excursions coming from a reflected Brownian Motion as explained at paragraph \ref{subsec:random-flipping}.

Let $\cal S$ be the space of real sequences $(a_k)_{k\in \mathbb N}$ and denote $\Phi : {\cal S}\times {\mathbb N} \rightarrow {\mathbb R}$ the coordinate function defined by
$\Phi((a_k)_{k\in \mathbb N} , n) = a_{n}$. From the independence of ${\cal H}$ and ${\cal F}_t^B$ and since 
$\pare{L_t^0(B^\beta)}_{t\geq 0}$ is adapted w.r.t. $\pare{{\cal F}_t^B}_{t\geq 0}$, from the properties of the conditional expectation,

\begin{equation*}
\begin{split}
&\1_{y\geq 0}\P^0\big[\, B^\beta_t \in dy ; L_t^0(B^\beta)\in dl    \, \big] = \1_{y\geq 0}\E^0 \left [\P^0\big[\, B^\beta_t \in dy ; L_t^0(B^\beta)\in dl    \,|\;{\cal F}^B\; \big]\right ]\\
&=\E^0 \left [\P^0\big[\, Y_{n_s(\rm \bf e)} > 0 ; |{\rm e}_s(t-\tau_{s-}(\omega), \omega)|\in dy ; L_t^0(B^\beta)\in dl    \,|\;{\cal F}^B\; \big]\right ]\\
&=\E^0 \left [\P^0\big[\,  \Phi((Y_{k})_{k\in {\mathbb N}}(\omega), n_s({\rm e}(\omega))) > 0 \,| \;{\cal F}^B\; \big];  |{\rm e}_s(t-\tau_{s-}(\omega), \omega)|\in dy  ;\,L_t^0(B^\beta)\in dl\right ]\\
&=\E^0 \left [\P^0\big[Y_n>0\big]  \mid_{n=n_s({\rm e}(\omega))}; |{\rm e}_s(t-\tau_{s-}(\omega), \omega)|\in dy ; L_t^0(B^\beta)\in dl  \right ]\\
&=\frac{1+\beta}{2}\P^0\big[\, |B^\beta_t| \in dy ; L_t^0(B^\beta)\in dl \big].
\end{split}
\end{equation*}
Proceeding similarly on $\R_-^*$ we get,
$$
\1_{y< 0}\P^0\big[\, B^\beta_t \in dy ; L_t^0(B^\beta)\in dl    \, \big]=\frac{1-\beta}{2} \P^0\big[\, -|B^\beta_t| \in dy ; L_t^0(B^\beta)\in dl \big],
$$
therefore the result.
\end{proof}

Using the two last lemmas we can prove the following result.

\begin{prop}
\label{densBL}
Let $W$ be a Brownian motion defined on $(C,\cC,\P)$ and $B^{\beta}$ the strong solution of \eqref{edsmu} with $\mu=0$.

We have for all $t>0,\;x\geq 0$,
$$
\begin{array}{lll}
\P^x\big[\, B^\beta_t \in dy ; L_t^0(B^\beta)\in dl    \, \big]&=&
\1_{y\geq 0}\1_{l>0}\frac{(1+\beta) (l+y+x)}{\sqrt{2\pi t^3}}\exp\big\{   -\frac{(l+y+x)^2}{2t}  \big\}dydl\\
\\
&&+\1_{y\geq 0}\frac{1}{\sqrt{2\pi t}}\big(  \exp\{-\frac{(y-x)^2}{2t}\} -   \exp\{-\frac{(y+x)^2}{2t}\}    \big)dy\delta_0(dl)\\
\\
&&+\1_{y<0}\1_{l\geq 0}\frac{(1-\beta) (l-y+x)}{\sqrt{2\pi t^3}}\exp\big\{   -\frac{(l-y+x)^2}{2t}  \big\}dydl.\\
\end{array}
$$

\end{prop}

\begin{rem} The result of  Proposition \ref{densBL}  appears as a corollary of a more general result stated in \cite{Zaitseva-2}. Note also that
the result of Proposition \ref{densBL} (and consequently the result stated in Proposition \ref{densitySBMmu}  in the next Section) differ slightly from results published by T. Appuhamillage et al. in the recent article \cite{waymire} where there is a computational error (see also \cite{gangsters} for a discussion). We will detail the computations for the sake of completeness and clarification.
\end{rem}

\begin{proof}[Proof of Proposition \ref{densBL}]
{\it Step 1.} Combining the results of the Lemmas \ref{lemsko} and \ref{excursion} we have 
$$
\begin{array}{lll}
\1_{y\geq 0}\P^0\big[\, B^\beta_t \in dy ; L_t^0(B^\beta)\in dl    \, \big]
&=&\frac{1+\beta}{2} \P^0\big[\, |B^\beta_t| \in dy ; L_t^0(B^\beta)\in dl \big]
\\
&=&\frac{1+\beta}{2} \P^0\big[\, |W_t| \in dy ; L_t^0(W)\in dl  \, \big].
\end{array}
$$

{\it Step 2.} Let  $x>0$. As $\1_{y\geq 0}\1_{l>0}\P^x[B^\beta_t \in dy ; L_t^0(B^\beta)\in dl ; t< \tau_0]=0$, we have, using the strong Markov property,
\begin{equation*}
\label{mar1}
\begin{array}{lll}
\1_{y\geq 0}\1_{l>0}\P^x[\, B^\beta_t \in dy ; L_t^0(B^\beta)\in dl    \, ]&=&
\1_{y\geq 0}\1_{l>0}\P^x[\, B^\beta_t \in dy ; L_t^0(B^\beta)\in dl ; t\geq \tau_0 \, ]
\\
\\
&=&\1_{y\geq 0}\1_{l>0}\E^x[\1_{\{ t\geq \tau_0 \}} \P^x[\, B^\beta_t \in dy ; L_t^0(B^\beta)\in dl \, | \cF_{\tau_0}]]
\\
\\
&=&\1_{y\geq 0}\1_{l>0}\int_0^t \P^0[  B^\beta_{t-s} \in dy ; L_{t-s}^0(B^\beta)\in dl \,]h(x,s)ds,
\\
\end{array}
\end{equation*}
where $h(x,.)$ is the density of $\tau_0$ under $\P^x$. But $h(x,.)$ is also the density of $T_0=\inf(t\geq 0~:~W_t =0)$. And using the first step of the proof we have
\begin{equation*}
\label{mar2}
\1_{y\geq 0}\1_{l>0}\P^0[  B^\beta_{t-s} \in dy ; L_{t-s}^0(B^\beta)\in dl\,]=
\1_{l>0}\frac{1+\beta}{2} \P^0[\, |W_{t-s}| \in dy ; L_{t-s}^0(W)\in dl  \, ].
\end{equation*}

Using again the strong Markov property we get
\begin{equation*}
\label{mar3}
\begin{array}{lll}
\1_{y\geq 0}\1_{l>0}\P^x[\, B^\beta_t \in dy ; L_t^0(B^\beta)\in dl    \, ]&=&
\1_{l>0}\E^x[\1_{\{ t\geq T_0 \}} \frac{1+\beta}{2}\P^x[\, |W_t| \in dy ; L_t^0(W)\in dl  \, | \cF_{T_0}]]
\\
\\
&=&\1_{l>0}\displaystyle\frac{1+\beta}{2}\P^x[\, |W_t| \in dy ; L_t^0(W)\in dl ; t\geq T_0 \, ]
\\
\\
&=&\1_{l>0}\displaystyle\frac{1+\beta}{2}\P^x[\, |W_t| \in dy ; L_t^0(W)\in dl \,]\\
\\
&=&\1_{y\geq 0}\1_{l>0}\displaystyle\frac{(1+\beta) (l+y+x)}{\sqrt{2\pi t^3}}\exp\big\{ -\frac{(l+y+x)^2}{2t}  \big\}dydl.
\end{array}
\end{equation*}

{\it Step 3.} It is a consequence of the reflection principle that

$$
\begin{array}{ll}
\1_{y\geq 0}\P^x[\, B^\beta_t \in dy ; L_t^0(B^\beta)=0    \, ] &=
\1_{y\geq 0}\P^x[\, B^\beta_t \in dy ; t<\tau_0    \, ]\\
\\
&=\1_{y\geq 0}\displaystyle\frac{1}{\sqrt{2\pi t}}\pare{  \exp\{-\frac{(y-x)^2}{2t}\} -   \exp\{-\frac{(y+x)^2}{2t}\}    }.\\
\end{array}
$$

\vspace{0.2cm}
Using Step 1 to 3 we have the result for $x\geq 0$ on $\R_+\times\R_+$. In order to retrieve the result on $\R_-^*\times\R_+$ we use Step 1 and 2 with $\frac{1+\beta}{2}$ replaced by $\frac{1-\beta}{2}$ and $\1_{y\geq 0}$ replaced by $\1_{y<0}$,
and the fact that for $x\geq 0$, $\1_{y<0}\P^x[\, B^\beta_t \in dy ; L_t^0(B^\beta)=0    \, ] =\1_{y<0}\P^x[\, B^\beta_t \in dy ; t<\tau_0    \, ]=0$.
\end{proof}

		\subsection{The law of the Skew Brownian motion with drift}
		We have the following proposition.

\begin{prop}
\label{densitySBMmu}  

We have that, 
$$p^{\beta,\mu}(t,x,y)=
\left\{ 
\begin{array}{l}
\frac{1}{\sqrt{2\pi t}}\exp\{\mu (y-x)-\frac{1}{2}\mu^2 t\}\big(  \exp\{ -\frac{(y-x)^2}{2t} \}  - \exp\{ -\frac{(y+x)^2}{2t} \}  \big)\\
+\,\,\frac{1+\beta}{\sqrt{2\pi t}}\exp\big\{ -\frac{(x+y)^2}{2t}+\mu (y-x) -  \frac{1}{2}\mu^2 t \big\}\\
\hspace{3,0 cm}\times \big[ 1 -  \beta\mu\sqrt{2\pi t} \exp\big\{ \frac{(x+y+t\beta\mu)^2 }{2t} \big\}  N^c(\frac{\beta\mu t+x+y}{\sqrt{t}}) \big],\\
\text{ if }x\geq 0,y\geq 0,\\
\\
\frac{1-\beta}{\sqrt{2\pi t}}\exp\big\{ -\frac{(x - y)^2}{2t} + \mu (y-x) -  \frac{1}{2}\mu^2 t \big\}\\
\hspace{3,0 cm}\times \big[ 1 -  \beta\mu\sqrt{2\pi t} \exp\big\{ \frac{(x - y+t\beta\mu)^2 }{2t} \big\}  N^c(\frac{\beta\mu t+x - y}{\sqrt{t}}) \big],\\
\text{ if }x\geq 0,y<0.\\
\\
\frac{1}{\sqrt{2\pi t}}\exp\{\mu (y-x)-\frac{1}{2}\mu^2 t\}\big(  \exp\{ -\frac{(y-x)^2}{2t} \}  - \exp\{ -\frac{(y+x)^2}{2t} \}  \big)\\
+\,\,\frac{1-\beta}{\sqrt{2\pi t}}\exp\big\{ -\frac{(x + y)^2}{2t} +\mu (y-x) -  \frac{1}{2}\mu^2 t \big\}\\
\hspace{3,0 cm}\times  \big[ 1 -  \beta\mu\sqrt{2\pi t} \exp\big\{ \frac{(-x-y+t\beta\mu)^2 }{2t} \big\}  N^c(\frac{\beta\mu t - y - x}{\sqrt{t}}) \big],\\
\text{ if }x<0,y<0,\\
\\
\frac{1+\beta}{\sqrt{2\pi t}}\exp\big\{ -\frac{(x-y)^2}{2t} +\mu (y-x) -  \frac{1}{2}\mu^2 t \big\} \\
\hspace{3,0 cm}\times\big[ 1 -  \beta\mu\sqrt{2\pi t} \exp\big\{ \frac{(y-x+t\beta\mu)^2 }{2t} \big\}  N^c(\frac{\beta\mu t+y-x}{\sqrt{t}}) \big],
\\
\text{ if }x<0,y\geq 0.\\
\end{array}
\right.
$$

\end{prop}

\begin{rem}
\label{rem-dens-pos}
It can be shown that the quantity $1 -  \beta\mu\sqrt{2\pi t} \exp\big\{ \frac{(|x|+|y|+t\beta\mu)^2 }{2t} \big\}  N^c(\frac{\beta\mu t+|x|+|y|}{\sqrt{t}})$
involved in $p^{\beta,\mu}(t,x,y)$ remains strictly positive, whatever the sign of $\beta\mu$ (see Remark \ref{sign-dens}).
\end{rem}

\begin{proof}[Proof of Proposition \ref{densitySBMmu}]

We have
$$
dB^{\beta,\mu}_t=dW^\mu_t+\beta dL^0_t(B^{\beta,\mu}),
$$
with $W^\mu_t=W_t+\mu t$ a Brownian motion under 
$\Q^\mu$ defined by $\dfrac{d\Q^\mu}{d\P}=\exp\{ -\mu W_t - \dfrac{1}{2}\mu^2t \}$.
Note that under $\Q^\mu$ the process $B^{\beta,\mu}$ starting from $0$ is distributed as $B^\beta$ starting from $0$ under $\P$.

For any bounded continuous function $f$ and any $t\geq 0$, we have
\begin{equation}
\label{changmu}
\begin{array}{lll}
\E_\P^x[f(B^{\beta,\mu}_t)]&=&\E_\P^0[f(B^{\beta,\mu}_t + x)]\\
&=&\E^0_{\Q^\mu}[f(B^{\beta,\mu}_t + x) \exp\{\mu W^\mu_t-\frac{1}{2}\mu^2t \} ]\\
\\
&=&\int\int_{\R^2}f(y + x)\exp\{\mu w-\frac{1}{2}\mu^2 t\}\P^0[B^\beta_t\in dy; \,W_t\in  dw]\\
&=&\int\int_{\R^2}f(y)\exp\{\mu w-\frac{1}{2}\mu^2 t\}\P^{x}[B^\beta_t\in dy; \,W_t-x\in  dw]\\
\end{array}
\end{equation}

Suppose $\beta>0$.

We set $\Phi_x(z,l)=(z, z - x-\beta l)$ which defines a bijection $\Phi_x:\R\times \R_+\to D_x$ where $D_x=\{(y,w)\in\R^2: y -x \geq w\}$.
Note that $(B^\beta_t,W_t-x)=\Phi_x(B^\beta_t,L^0_t(B^\beta))$. Besides, almost surely, $(B^\beta,L^0(B^\beta))\in \R\times\R_+$ and $(B^\beta,W-x)\in D_x$.

For $x>0$, Proposition \ref{densBL} ensures that
the measure $\P^x[\, B^\beta_t \in dy ; L_t^0(B^\beta)\in dl  \, ]$ has a density with respect to $dy\,dl$ on $\R\times \R^{*,+}$, and gives mass to the segments of $\R_{+}\times\{ 0\}$ with the  density
\begin{equation}
\label{densseg}
\P^x[\, B^\beta_t \in dy ; L_t^0(B^\beta)=0  \, ]=\frac{1}{\sqrt{2\pi t}}\big(  \exp\{-\frac{(y-x)^2}{2t}\} -   \exp\{-\frac{(y+x)^2}{2t}\}    \big)dy.
\end{equation}

Let us denote $\Delta_x:=\{(y,w)\in\R_{+}\times \R: y = w+x\}=\Phi_x(\R_+\times\{0\})$. The measure $\P^{x}[B^\beta_t\in dy; \,W_t-x\in  dw]$ has a density $g^{x}_{B^\beta,W}(y,w)$
 with respect to $dy\, dw$ on $D_x\setminus \Delta_x=\Phi_x(\R\times\R^{*,+})$. 
But it gives mass to the segments of the line $\Delta_x$. Let us denote $\Phi_x^{-1}(y,w):=(\Phi^{-1}_1(y,w),\Phi^{-1}_2(y,w))$ and notice that $\Phi^{-1}_1(y,w)=y$. Let $A_1\subset\R_{+}$ and $A=\{(y,w)\in\R^2:y\in A_1,y=w+x\}\subset\Delta_x$. As $\Phi_x^{-1}(A)\subset\R_+\times\{ 0\}$ we have
$$
\begin{array}{lll}
\P^{x}[(B^\beta_t,W_t-x)\in A]&=&\P^x[(B^\beta_t,L^0_t(B^\beta))\in \Phi_x^{-1}(A)]\\
&=&\P^x[B^\beta_t\in\Phi^{-1}_1(A); L_t^0(B^\beta)=0  \, ]\\
&=&\P^x[B^\beta_t\in A_1; L_t^0(B^\beta)=0  \, ].\\
\end{array}
$$
Using this and \eqref{densseg} in \eqref{changmu} we get
$$
\begin{array}{lll}
\E_\P^x[f(B^{\beta,\mu}_t)]&=&
\int\int_{D_x\setminus\Delta_x}f(y)\exp\{\mu w-\frac{1}{2}\mu^2 t\}\P^{x}[B^\beta_t\in dy; \,W_t-x\in  dw]\\
&&+\int\int_{\Delta_x} f(y)\exp\{\mu w-\frac{1}{2}\mu^2 t\}\P^{x}[B^\beta_t\in dy; \,W_t-x\in  dw]\\
\\
&=&\int_{\R}f(y) \int_{-\infty}^{y-x} \exp\{\mu w-\frac{1}{2}\mu^2 t\}g^{x}_{B^\beta,W}(y,w)dw \,dy\\
&&+\int_{\R_{+}} f(y)\frac{1}{\sqrt{2\pi t}}\exp\{\mu (y-x)-\frac{1}{2}\mu^2 t\}\big(  \exp\{-\frac{(y-x)^2}{2t}\} -   \exp\{-\frac{(y+x)^2}{2t}\}    \big)dy.
\end{array}
$$
We now compute $ \int_{-\infty}^{y-x} \exp\{\mu w-\frac{1}{2}\mu^2 t\}g^{x}_{B^\beta,W}(y,w)dw$ with a change of variable and an integration by parts. We have for $y\geq 0$,
$$
 \int_{-\infty}^{y-x} \exp\{\mu w-\frac{1}{2}\mu^2 t\}g^{x}_{B^\beta,W}(y,w)dw=
 \frac{e^{{-\frac{1}{2}\mu^2 t}}}{\beta}\int_{-\infty}^{y-x}e^{\mu w}\dfrac{(1+\beta)(\frac{y-w-x}{\beta}+x+y)}{\sqrt{2\pi t^3}}e^{  -\frac{ (\frac{y-w -x}{\beta}+x+y)^2 }{2t}   }dw.
$$
And,
$$
\begin{array}{l}
\int_{-\infty}^{y-x}e^{\mu w}(\frac{y-w-x}{\beta}+x+y)e^{  -\frac{ (\frac{y-w-x }{\beta}+x+y)^2 }{2t}   }dw
=\beta e^{\mu(y-x)}\int_0^\infty e^{-\beta \mu w'}(w' +x +y)e^{-\frac{(w'+x+y)^2}{2t}}dw'\\
\\
=\beta e^{\mu(y-x)}\pare{te^{-\frac{(x+y)^2}{2t}} -\beta\mu t \int_0^\infty e^{-\beta\mu w' - \frac{(w'+x+y)^2}{2t}}dw'} \\
=\beta e^{\mu(y-x)}\pare{te^{-\frac{(x+y)^2}{2t}} -\sqrt{2\pi}\beta\mu t^{3/2}e^{{\frac{\beta^2}{2}\mu^2 t}}e^{\beta \mu (x+y)} N^c(\frac{x+y+t\beta\mu}{\sqrt{t}})}\\
= \beta\, t\,e^{\mu(y-x)}e^{-\frac{(x+y)^2}{2t}}\pare{1 -\sqrt{2\pi t}\beta\mu e^{\frac{(x+y+\beta\mu t)^2}{2t}} N^c(\frac{x+y+t\beta\mu}{\sqrt{t}})}
\end{array}
$$
which yields the desired result. The cases $y<0$ and $\beta < 0$ are treated in a similar way.

For the case $x<0$, we perform the change of variable $x\rightarrow -x$, $y\rightarrow -y$, $\beta\rightarrow -\beta$ and $\mu\rightarrow -\mu$.

\end{proof}	    
	    
		\section{Exact simulation of bridges of a Skew motion with drift}
\label{ss:section:bridge}

For $0<t<T$ let us denote $q^{\beta,\mu}(t,T,a,b,y)$ the probability  density of $B^{\beta,\mu}_t$ knowing that $B^{\beta,\mu}_0=a$ and $B^{\beta,\mu}_T=b$. That is to say
$$
\P[ B^{\beta,\mu}_t\in dy\,|\, B^{\beta,\mu}_0=a\, , \, B^{\beta,\mu}_T=b]=q^{\beta,\mu}(t,T,a,b,y)dy.$$

Note that with these notations, $q^{0,0}(t,T,a,b,y)$ is the probability density of $W_t$ knowing that $W_0=a$ and $W_T=b$.  As the law of the Brownian bridge is well known, sampling from $q^{0,0}(t,T,a,b,y)$ is easy.
\vspace{0.5cm}

 In order to sample along the law given by $q^{\beta,\mu}(t,T,a,b,y)$ with a rejection algorithm using Brownian bridges values as proposals, we will use the two following results.

\begin{lem}
\label{explicicte-densbridge}
Let $a,b\in\R$, $0<t<T$. 

For $(\beta,\mu)\in (-1,1)\times\R$, we have
\begin{equation}
\label{formula-densbridge}
\forall y\in\R,\quad q^{\beta,\mu}(t,T,a,b,y)  = \frac{p^{\beta,\mu}(t,a,y)p^{\beta,\mu}(T-t,y,b)}{p^{\beta,\mu}(T,a,b)}.
\end{equation}
\end{lem}

\begin{proof}

This comes from the Markov property for solutions of \eqref{edstl} (see Subsubsection \ref{sssection-markov} and the references therein).

\end{proof}

\begin{lem}
We have for all $\mu\in\R$, 
$$\forall 0<t<T,\;\forall a,b,y\in\R,\quad q^{0,\mu}(t,T,a,b,y)=q^{0,0}(t,T,a,b,y).$$
\end{lem}

\begin{proof}
It is a direct consequence of Equations \eqref{def-pmu}  and \eqref{formula-densbridge}.
\end{proof}

\begin{lem}
\label{majmindens}
Let $a,b\in\R$, $0<t<T$. 

\begin{list}{--}{}
\item For $(\beta,\mu)\in (-1,1)\times\R$ with $\beta\mu\geq 0$, we have
\begin{equation}
\label{majpontpos}
\forall y\in\R,\quad q^{\beta,\mu}(t,T,a,b,y) \leq K^{\beta,\mu}_{T,a,b}\,q^{0,0}(t,T,a,b,y),
\end{equation}
where
$$
K^{\beta,\mu}_{T,a,b}:= 4 \overline{\alpha}^2\frac{p^{0,\mu}(T,a,b)}{p^{\beta,\mu}(T,a,b)}, 
$$ 
and $\overline{\alpha}:=\max(\frac{1+\beta}{2},\frac{1-\beta}{2})$. 

\item For $(\beta,\mu)\in (-1,1)\times\R$ with $\beta\mu < 0$, set 
$$
\gamma^{\beta,\mu}(t,z):=1-\beta\mu\sqrt{2\pi t}\exp(\frac{(z+t\beta\mu)^2}{2t})N^c(\frac{\beta\mu t + z}{\sqrt{t}}).
$$
Then,
\begin{equation}
\label{majdens2}
p^{\beta,\mu}(t,x,y)\leq 2\overline{\alpha}\gamma^{\beta,\mu}(t,|x|)p^{0,\mu}(t,x,y)
\end{equation}
and
\begin{equation}
\label{majdens3}
p^{\beta,\mu}(t,x,y)\leq 2\overline{\alpha}\gamma^{\beta,\mu}(t,|y|)p^{0,\mu}(t,x,y).
\end{equation}
In particular,
\begin{equation}
\label{majpontneg}
\forall y\in\R,\quad q^{\beta,\mu}(t,T,a,b,y) \leq K^{\beta,\mu}_{t,T,a,b}\,q^{0,0}(t,T,a,b,y),
\end{equation}
where
$$
K^{\beta,\mu}_{t,T,a,b}:= 4\overline{\alpha}^2\gamma^{\beta,\mu}(t,|a|)\gamma^{\beta,\mu}(T-t,|b|)\frac{p^{0,\mu}(T,a,b)}{p^{\beta,\mu}(T,a,b)}.
$$ 
\end{list}

\end{lem}

\begin{proof}

{\it Case $\beta\mu\geq 0$.} 

Let  $t>0$ and $x\geq 0$. Looking at Proposition \ref{densitySBMmu} it is clear that for $y< 0$,
$$p^{\beta,\mu}(t,x,y)\leq (1+\beta)p^{0,\mu}(t,x,y).$$

For $y\geq 0$ we have
$$
\begin{array}{lll}
p^{\beta,\mu}(t,x,y)& \leq & p^{0,\mu}(t,x,y) + \frac{\beta}{\sqrt{2\pi t}}\exp\big\{ -\frac{(x+y)^2}{2t}+\mu(y-x)-\frac{1}{2}\mu^2t  \big\}\\
\\
&\leq &(1+\beta)\,p^{0,\mu}(t,x,y),\\
\end{array}
$$
where we have used $(y-x)^2\leq (y+x)^2$ (because $x,y>0$).
We can proceed in a similar way for $x<0$ and finally, we get that
\begin{equation}
\label{majdens}
\forall t>0,\;\forall x,y\in\R,\quad p^{\beta,\mu}(t,x,y) \leq  2\overline{\alpha}\,p^{0,\mu}(t,x,y).
\end{equation}

Thus, using the previous inequality gives 
\begin{equation}
\label{calcmajpont}
\begin{split}
q^{\beta,\mu}(t,T,a,b,y)  &= \frac{p^{\beta,\mu}(t,a,y)p^{\beta,\mu}(T-t,y,b)}{p^{\beta,\mu}(T,a,b)}\\
&\leq  4\overline{\alpha}^2\frac{p^{0,\mu}(T,a,b)}{p^{\beta,\mu}(T,a,b)}\frac{p^{0,\mu}(t,a,y)p^{0,\mu}(T-t,y,b)}{p^{0,\mu}(T,a,b)}\\
&\leq 4 \overline{\alpha}^2\frac{p^{0,\mu}(T,a,b)}{p^{\beta,\mu}(T,a,b)}q^{0,0}(t,T,a,b,y).
\end{split}
\end{equation}

\vspace{0.5cm}

{\it Case $\beta\mu< 0$.} 
Let us denote $\Gamma^{\beta,\mu}(t,x,y):=1 -  \beta\mu\sqrt{2\pi t} \exp\big\{ \frac{(|x|+|y|+t\beta\mu)^2 }{2t} \big\}  N^c(\frac{\beta\mu t+|x|+|y|}{\sqrt{t}})$.
For fixed $x\in\R$, $y\mapsto\Gamma^{\beta,\mu}(t,x,y)$ is an even function. As we have
\begin{equation}
\label{majgauss}
\forall z>0,\quad z\, e^{\frac{z^2}{2}}\int_z^\infty e^{-\frac{u^2}{2}}du< 1,
\end{equation} 
the function $z\mapsto \sqrt{2\pi}\exp(\frac{z^2}{2})N^c(z)$ has negative first derivative on $\R^+$. Therefore $y\mapsto\Gamma^{\beta,\mu}(t,x,y)$ is decreasing on $\R^+$
and we have $\max_{y\in\R}\Gamma^{\beta,\mu}(t,x,y)=\gamma^{\beta,\mu}(t,|x|)$.
Using this and the same kind of computations than in the previous case we get \eqref{majdens2}. As the roles of $x$ and $y$ are symmetric in 
$\Gamma^{\beta,\mu}(t,x,y)$, we get \eqref{majdens3}. We then obtain \eqref{majpontneg}, using the same computations than for \eqref{calcmajpont}.
\end{proof}

\begin{rem}
\label{sign-dens}
Note that \eqref{majgauss}  also allows to prove that $p^{\beta,\mu}(t,x,y)$ remains strictly positive (see Remark \ref{rem-dens-pos}).

\end{rem}		
		
	Lemma \ref{majmindens} suggests to use the following rejection algorithm in order to sample along $q^{\beta,\mu}(t,T,a,b,y)dy$ (see for example Proposition 1 in \cite{beskos1}).
		
		\vspace{0.3cm}
		
		\begin{center}
\hrulefill

		\textsc{  
		\begin{center}
EXACT SIMULATION ALGORITHM ALONG $q^{\beta,\mu}(t,T,a,b,y)dy$
\end{center}
\begin{enumerate}
\item Set $K=K^{\beta,\mu}_{T,a,b}$ if $\beta\mu\geq 0$, $K=K^{\beta,\mu}_{t,T,a,b}$ otherwise.
\item Sample $Y$ along $q^{0,0}(t,Ta,b,y)dy$.
\item Evaluate 
$$f(Y)\defi \frac{1}{K}\frac{q^{\beta,\mu}(t,T,a,b,Y)}{q^{0,0}(t,T,a,b,Y)}\leq 1.$$
\item Simulate $U\sim\mathcal{U}([0,1])$.
If $U\leq f(Y)$ accept the proposed value $Y$. Else return to Step 2.
\end{enumerate}  
}
	\hrulefill
	\end{center}

\section{Numerical experiments}
\label{sec-num}	

{\bf Example 1.} We first deal with a toy example. We consider the following SDE
\begin{equation}
\label{edscos}
dX_t=dW_t-\frac{\pi}{2}\cos(\frac{\pi}{5}X_t)dt+\beta dL^0_t(X),\quad X_0=x_0,
\end{equation}
with $\beta=0.6$, and $x_0=0.2$. Note that, here,  the drift $\bar{b}(x)=-\frac{\pi}{2}\cos(\frac{\pi}{5}x)$ is bounded and of class $C^\infty$ on the whole real line.

For the exact procedure  the constant drift involved in 
Subsection \ref{sec:algo_nondege}, equals $\mu=\bar{b}(0)=-\frac{\pi}{2}$. So  we will have to sample bridges of SBM with non zero drift $\mu$, using the results of Section \ref{ss:section:bridge}.

 We have first to sample $X_T$ from 
$$
\begin{array}{lll}
h(y)&=&C\exp\pare{B(y)-B(x_0)} p^{\beta,\mu}(T,x_0,y)\\
\\
&=&C\exp\pare{\frac{5}{2}\big(\sin(\frac{\pi}{5}x_0)-\sin(\frac{\pi}{5}y)\big)-\mu(y-x_0)  } p^{\beta,\mu}(T,x_0,y)\\
\end{array}$$
(Step 1 of the Exact Algorithm). This can be done by rejecting standard normal random variables with mean $x_0$. Indeed, using \eqref{majdens2}, we have here
$$
\frac{h(y)}{C}\leq 2\overline{\alpha}\gamma^{\beta,\mu}(T,|x_0|)\exp(5-\frac{\mu^2T}{2})\,p^{0,0}(T,x_0,y).$$
Then we accept or reject the proposed value $X_T$, using Steps 2 to 4 of the Exact Algorithm, with bridges of $B^{\beta,\mu}$,
$$\tilde{\phi}(x)=\frac{\pi^2}{8}\cos^2(\frac{\pi}{5}x)+\frac{\pi^2}{20}\sin(\frac{\pi}{5}x)+\frac{\pi^2}{20},$$
and $K=\frac{9\pi^2}{20}$ as un upper bound for $\tilde{\phi}$.

We plot on Figure \ref{fig1} (top and bottom figures) the approximated density obtained with $10^6$ simulations of $X_T$, sampled with our exact procedure. On the top figure we plot the approximated densities
obtained with $10^6$ simulations of the Euler Scheme used in \cite{martinez04a} and \cite{martinez06a}, for decreasing time steps. We can observe the convergence of Euler type simulations to exact ones. Note that to have the Euler scheme fitting the exact procedure we have to take a fine time step (namely 
$\Delta t=10^{-4}$). 
This is because, as shown in
\cite{martinez04a}, the rate of weak convergence of the Euler scheme in this situation is of order $(\Delta t)^{1/2-\epsilon}$, for a smooth initial condition. 


\begin{figure}
\begin{center}
\epsfig{file=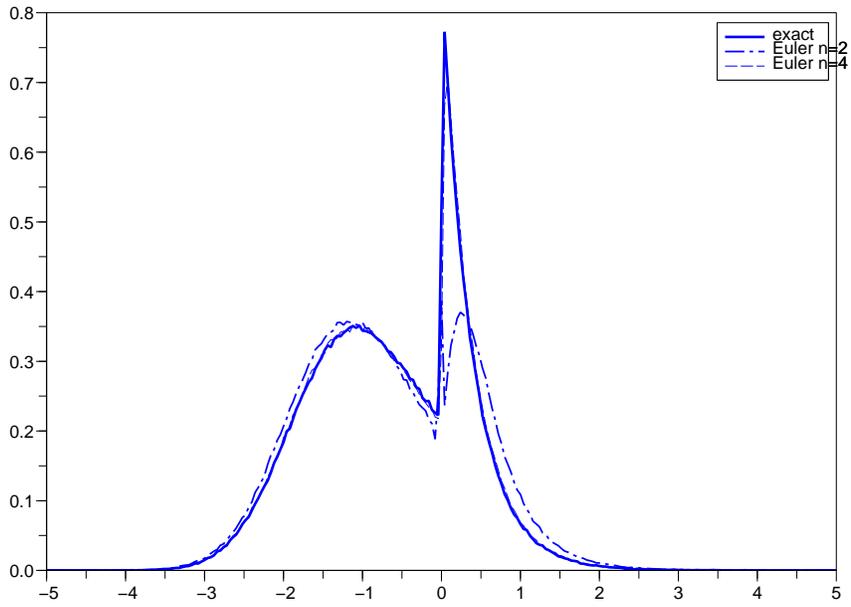, width=14cm}
\epsfig{file=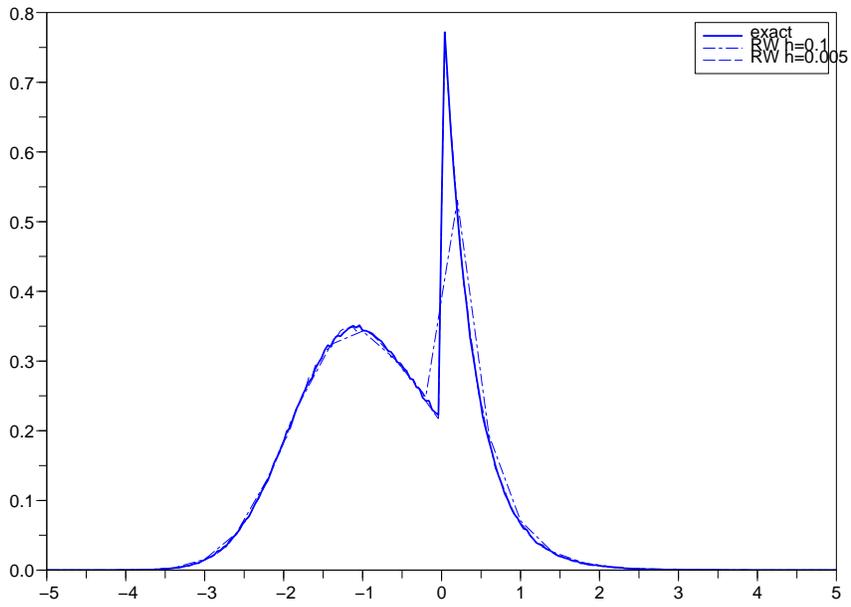, width=14cm}
\caption{ Approximated densities of the positions at time $T=1.0$ of $10^6$ paths of the solution of \eqref{edscos} starting from $x_0=0.2$: exact versus Euler with time step $\Delta t=10^{-n}$, for $n=2,4$ (top) and exact versus random walk with space steps $h=\frac{1}{10},\frac{1}{200}$ (bottom).}
\label{fig1}
\end{center}
\end{figure}

On the bottom figure the approximated density is compared with the approximated densities obtained with $10^6$ simulations
of the random walk based method studied in \cite{etore-lejay06a}, for decreasing space steps. Again we can observe the convergence of
the process with discretization error.

In Table \ref{tabrej1} we report the empirical acceptance ratios for the rejection step using $\tilde{\phi}$ and the Poisson point process in the Exact Algorithm (this corresponds to the column Exact Algorithm in the table), 
and for the rejection sampling of bridges of the SBM with drift (this is the average acceptance ratio in this case).

\begin{table}
 \begin{center}
\begin{tabular}{ccc}
\hline
 & Exact Algorithm  &  Bridges  \\
 \hline
 \\
Acceptance Ratio&  0.28   &   0.18 \\
\\
\hline
\end{tabular}
\caption{Acceptance ratios in Example 1.}
\label{tabrej1} 
\end{center}
\end{table}

In Table \ref{tab1} we report the CPU times needed to get the $10^6$ simulations, with the three different methods (and with the different discretization steps we have used).
Programs were written in C-language and executed on a personal computer equipped with an Intel Core 2 duo processor, running at $2.23$ Ghz.

On this example the exact simulation is competitive, compared to schemes with very fine grids.

\begin{table}
 \begin{center}
\begin{tabular}{ccc}
\hline
Exact & Euler  &  Random Walk  \\
& ($\Delta t=10^{-n}$, $n=2,4$)& ($h=\frac{1}{10},\frac{1}{200}$)\\
\\
\hline
239s& 17s&    3.52s\\
&  1680s& 1411s\\
\hline
\end{tabular}
\caption{CPU times for $10^6$ simulations of $X_T$.}
\label{tab1} 
\end{center}
\end{table}

\vspace{0.8cm}
{\bf Example 2.} 
 We want now to sample along the law of the continuous Markov process $X$ generated by
\begin{equation}
\label{L}
L=\frac{1}{2}\frac{\mathrm{d}}{{dx}}\big(a\frac{\mathrm{d}}{{dx}}\cdot\big)
\end{equation}
with
$$
a(x)=\left\{
\begin{array}{ll}
\frac{x^2+x+1}{(2x+1)^2}&\text{if }x\geq 0\\
\\
\frac{3x^2-x+2}{(6x-1)^2}&\text{if }x< 0.\\
\end{array}
\right.
$$
Note that $a(0+)=1\neq 2=a(0-)$. The coefficient $a(x)$ is of class $C^1$ on $\R^{*,-}$ and $\R^{*,+}$, and uniformly strictly positive and bounded, which ensures the existence of $X$; in addition  $X$ solves
\begin{equation}
\label{edsop1}
dX_t=\sqrt{a(X_t)}dW_t+\frac{a'(X_t)}{2}dt+\frac{a(0+)-a(0-)}{a(0+)+a(0-)}dL_t^0(X),
\end{equation}
(see \cite{lejay-martinez04a}, \cite{etore05a}). We define the Lamperti transformation $\Phi(x)=\int_0^xdz/\sqrt{a(z)}$ and set 
$Y_t~:=~\Phi(X_t)$. Then 
\begin{equation}
\label{edsop2}
dY_t=dW_t+\frac{1}{2}(\sqrt{a})'\circ\Phi^{-1}(Y_t)dt+\frac{\sqrt{a(0+)}-\sqrt{a(0-)}}{\sqrt{a(0+)}+\sqrt{a(0-)}}dL_t^0(Y),
\end{equation}
(this follows from Proposition 3.1 in \cite{etore05a}; see also \cite{lejay-martinez04a} and \cite{ouknine}). Firstly, note that
$\Big|\frac{\sqrt{a(0+)}-\sqrt{a(0-)}}{\sqrt{a(0+)}+\sqrt{a(0-)}}\Big|<1$. Secondly, we have
$$
(\sqrt{a})'(x)=\left\{
\begin{array}{ll}
\frac{1}{2\sqrt{x^2+x+1}}-2\frac{\sqrt{x^2+x+1}}{(2x+1)^2}&\text{if }x\geq 0\\
\\
-\frac{1}{2\sqrt{3x^2-x+1}}+6\frac{\sqrt{3x^2-x+2}}{(6x-1)^2}&\text{if }x< 0,\\
\end{array}
\right.
$$

$$
\Phi(x)=\left\{
\begin{array}{ll}
2\sqrt{x^2+x+1}-2&\text{if }x\geq 0\\
\\
-2\sqrt{3x^2-x+1}+2\sqrt{2}&\text{if }x< 0,\\
\end{array}
\right.
\hspace{0,3 cm}
\text{and}
\hspace{0,3 cm}
\Phi^{-1}(y)=\left\{
\begin{array}{ll}
\frac{-1+\sqrt{(y+2)^2-3}}{2} &\text{if }y\geq 0\\
\\
\frac{1-\sqrt{1-12[2-(\sqrt{2}-y/2)^2]}}{6}&\text{if }y< 0.\\
\end{array}
\right.
$$
As $(\sqrt{a})'(x)$ is bounded with bounded first derivative on $\R^{*,-}$ and $\R^{*,+}$, the explicitly known coefficients $\beta=\frac{\sqrt{a(0+)}-\sqrt{a(0-)}}{\sqrt{a(0+)}+\sqrt{a(0-)}}$
and $\bar{b}(y)=\frac{1}{2}(\sqrt{a})'\circ\Phi^{-1}(y)$  satisfy the assumptions of Subsubsection~\ref{assumptions}. Thus we can perform exact sampling
from \eqref{edsop2}, and, applying the exact inverse transformation $\Phi^{-1}$, get samples from \eqref{edsop1} with absolutely no discretization error.

Here we have,
$$\mu=\frac{1}{4}\frac{a'(0+)-a'(0-)}{\sqrt{a(0+)}-\sqrt{a(0-)}}  =-\frac{26}{4(1-\sqrt{2})}.$$

As we have
$$B(y)=
\left\{
\begin{array}{ll}
-\mu y+\frac{1}{2}\log(\sqrt{a}\circ\Phi^{-1}(y))&\text{if }y\geq 0\\
\\
-\mu y+\frac{1}{2}[\log(\sqrt{a}\circ\Phi^{-1}(y))-\log(\sqrt{2})]&\text{if }y< 0,\\
\end{array}
\right.
$$
we can show (using again \eqref{majdens2}), that for all $y_0,y\in\R$ and $T>0$,
$$
\exp\pare{B(y)-B(y_0)} p^{\beta,\mu}(T,y_0,y)\leq \sqrt{\sqrt{2}}\sqrt{\sqrt{24}}\,2\overline{\alpha}\gamma^{\beta,\mu}(T,|y_0|)\,e^{-\frac{1}{2}\mu^2T}p^{0,0}(T,y_0,y).$$

This allows to sample $Y_T$ from
$h(y)=C\exp\pare{B(y)-B(y_0)} p^{\beta,\mu}(T,y_0,y)$, by rejecting normal variables with mean $y_0$ and variance $T$.

We then accept or reject the proposed value $Y_T$ by using bridges of $B^{\beta,\mu}$ and
$$
\tilde{\phi}(y)=\frac{((1/2)(\sqrt{a})'\circ\Phi^{-1}(y))^2+(1/2)((\sqrt{a})''\sqrt{a})\circ\Phi^{-1}(y)}{2},$$
with
$$
(\sqrt{a})''(x)
\left\{
\begin{array}{ll}
 -\frac{2x+1}{4(x^2+x+1)^{3/2}}-\frac{1}{(2x+1)\sqrt{x^2+x+1}}+8\frac{\sqrt{x^2+x+1}}{(2x+1)^3}&\text{if }x\geq 0\\
\\
  \frac{6x-1}{4(3x^2-x+2)^{3/2}}-\frac{3}{(2x+1)\sqrt{3x^2-x+2}}+72\frac{\sqrt{3x^2-x+2}}{(6x-1)^3}&\text{if }x< 0.\\
\end{array}
\right.
$$

We take $K=\frac{(6\sqrt{2}-1/2)^2/4+(141-1/8)/2}{2}$ as an upper bound for $\tilde{\phi}$. 
We plot on Figure \ref{fig3}
the approximated density computed with $10^7$ simulations of $X_T$ for $x_0=0.0$ and $T=1$, obtained from the exact procedure. We plot on the same figure
the approximated densities obtained with the Euler scheme and the random walk approximation mentioned in Example 1.

\begin{figure}
\begin{center}
\epsfig{file=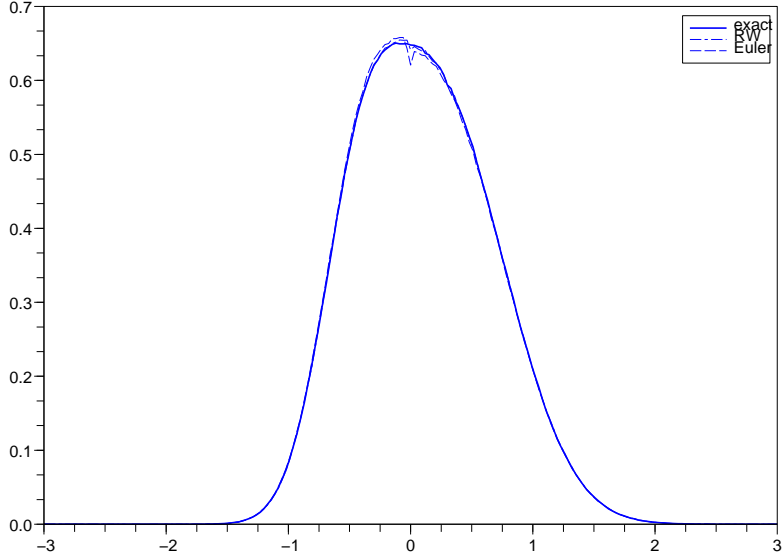,width=13cm}
\caption{DIVERGENCE FORM OPERATOR: Approximated density of the positions at time $T=1$ of $10^7$ paths of the solution of \eqref{edsop1} starting from $x_0=0.0$: exact versus random walk with space step $h=3.10^{-3}$ and Euler scheme with $\Delta t=10^{-4}$.}
\label{fig3}
\end{center}
\end{figure}

We report in Table \ref{tabrej2} the acceptance ratios.

\begin{table}
 \begin{center}
\begin{tabular}{cccc}
\hline
 & Exact Algorithm  &  Bridges  \\
 \hline
 \\
Acceptance Ratio&  0.017   &   0.5 \\
\\
\hline
\end{tabular}
\caption{Acceptance ratios in Example 2.}
\label{tabrej2} 
\end{center}
\end{table}

\begin{rem}
Note that the acceptance ratio for the algorithm in the first example in Table \ref{tabrej1} is quite low but decreases to less than $2\%$ in Table \ref{tabrej2} in the context of the second example. These figures are closely related to the measurement of the "distance" between the measure of the initial process from the reference measure and so these limitations of the algorithm arise even in the "classical" setting of the reference article \cite{beskos2} (for example with a rapidly varying drift).
Nevertheless, in terms of CPU time, the performance of the algorithm seems quite competitive, in comparison with those of discretization schemes.
\end{rem}

\begin{rem}
Note that, at least graphically and contrary to what we can see on Figure \ref{fig1}, the transition density plotted on Figure \ref{fig3} seems to be continuous at $0$~:~this matches the well-known theoretical result, which asserts that the transition density of diffusion semigroups corresponding to elliptic divergence form operator of the form \eqref{L} is always continuous. We refer to Stroock \cite{stroock} for a proof based on the self-adjoint properties of these semi-groups and Nash's inequality.
\end{rem}

\section{Discussion and concluding remarks}
\label{sec-conclu}

		\subsection{An open problem : the path decomposition of a skew Brownian bridge}
	An important issue for the extension of the initial exact simulation method is to overcome the restraining assumptions made on the drift function 
$\bar{b}$ (see Section \ref{assumption-1}) : namely, the assumption of boundedness for $\bar{b}$. 

For example, it is frustrating that these assumptions do not allow us to simulate exactly what one may call the "Skewed Ornstein-Uhlenbeck"
diffusion process. This difficulty appears even in the classical case (solutions of non skewed SDEs) and the fundamental reason is that we do not know how to simulate exactly a Poisson Point Process with $\sigma$-finite intensity on the whole space $\R_+\times\R$.

In the classical case,  where $\bar{b}$ is everywhere differentiable and no local time is involved (non skewed SDEs), this problem is solved by decomposing the trajectory of the standard Brownian bridge on $[0,T]$ w.r.t. the space-time point where it attains its maximum or both its maximum and its minimum : we refer to \cite{beskos5} for a detailed presentation of this problem in the classical setting.
 
Consequently, if one wants to overcome the restraining assumptions in Section \ref{assumption-1} concerning the drift function 
$\bar{b}$, one has to search for such kind of decompositions for (at least) the Skew Brownian Motion (not to mention the drifted Skew Brownian Motion). Up to our knowledge, no results can be found in the literature concerning this decomposition and this open problem seems difficult to us. However, we give below some insight concerning this problem thanks to an application of a theoretical result stated in  \cite{Pitman-Yor}.

Let $\tau^{\beta,\mu}_{z} : =\inf(s\geq 0~:~B^{\beta,\mu}_s = z)$. Set $u_\lambda(x ; z):={\mathbb E}^x\pare{{\rm e}^{-\lambda \tau^{\beta,\mu}_{z}}}$ which gives the Laplace transform of $\tau^{\beta,\mu}_{z}$ at $\lambda > 0$ (with $B^{\beta,\mu}_0=x$).
\begin{prop}  \text{(case $\mu = 0$)}

In the simple case where $\mu =0$, the function $u_\lambda$ is given by
\begin{equation}
\label{hitting-time-distrib}
u_\lambda(x;z) =
\left \{
\begin{array}{lll}
&\displaystyle \frac{{\rm sinh}\pare{\sqrt{2\lambda}(z-x)}}{{\rm sinh}(\sqrt{2\lambda}z)}\frac{1+\beta}{2\cosh(\sqrt{2\lambda}z)-(1-\beta){\rm e}^{-\sqrt{2\lambda}z}}\\
 &\displaystyle \hspace{0.2 cm} +\,\,\frac{{\rm sinh}(\sqrt{2\lambda}x)}{{\rm sinh}(\sqrt{2\lambda}z)}&\hspace{0.3 cm}\text{if}\hspace{0.2 cm}z\geq x >0,\\
 \\
&\displaystyle {\rm e}^{-\sqrt{2\lambda}(x-z)}&\hspace{0.3 cm}\text{if}\hspace{0.2 cm}x\geq z\geq 0,\\
&\displaystyle {\rm e}^{\sqrt{2\lambda}x}\frac{1+\beta}{2\cosh(\sqrt{2\lambda}z)-(1-\beta){\rm e}^{-\sqrt{2\lambda}z}}&\hspace{0.3 cm}\text{if}\hspace{0.2 cm}x< 0< z,\\
&\displaystyle {\rm e}^{\sqrt{2\lambda}(x-z)}&\hspace{0.3 cm}\text{if}\hspace{0.2 cm}0\geq z\geq x,\\
&\displaystyle {\rm e}^{-\sqrt{2\lambda}x}\frac{1-\beta}{2\cosh(\sqrt{2\lambda}z)-(1+\beta){\rm e}^{\sqrt{2\lambda}z}}&\hspace{0.3 cm}\text{if}\hspace{0.2 cm}z< 0< x,\\
\\
&\displaystyle\frac{{\rm sinh}\pare{\sqrt{2\lambda}(z-x)}}{{\rm sinh}(\sqrt{2\lambda}z)}\frac{1-\beta}{2\cosh(\sqrt{2\lambda}z)-(1+\beta){\rm e}^{\sqrt{2\lambda}z}}\\
&\displaystyle \hspace{0.2 cm} +\,\, \frac{{\rm sinh}(\sqrt{2\lambda}x)}{{\rm sinh}(\sqrt{2\lambda}z)}&\hspace{0.3 cm}\text{if}\hspace{0.2 cm}0> x\geq z.
\end{array}
\right .
\end{equation}
\end{prop} 
\begin{rem}
Note that if $\beta =0$, we retrieve after easy computations the well known result that gives the Laplace Transform of the law of the hitting time of $z$ by a standard Brownian Motion starting from $x$.  
\end{rem}
\begin{proof}
We only sketch the proof. The different cases may be easily conjectured from the description of the excursion measure for the SBM $(B_s^{\beta,0})_{s\geq 0}$ and the known facts concerning the standard Brownian Motion (decomposition of the different cases when a skew Brownian Motion reaches $z$ starting from $x$). In order to
check rigorously the validity of the result, one may verify that the formulas (\ref{hitting-time-distrib}) yield a solution of Dynkin's problem associated to the generator of $(B_s^{\beta,0})_{s\geq 0}$ namely~:~
\begin{equation}
\left \{
\begin{array}{l}
\displaystyle \frac{1}{2}\frac{d^2 }{dx^2}u_\lambda(.;z)= \lambda u_\lambda(.;z)\\
\displaystyle u_\lambda(z;z)=1,
\end{array}
\right .
\end{equation}
with
$$u_\lambda(.;z) \in \{g\in C^0({\mathbb R})\cap C^2\pare{(-\infty,0)\cup (0,\infty)}~:~(1+\beta) g'(0+) = (1-\beta)g'(0-)\}.$$

\end{proof}

A scale function $s$ and the corresponding integrated speed measure $m$ of a Skew Brownian Motion are given by
$$
s(x)=\left \{
\begin{array}{l}
\frac{2}{\beta + 1}x\hspace{0.3 cm}\text{ if }x\geq 0\\
\frac{2}{1-\beta}x\hspace{0.3 cm}\text{ if }x< 0
\end{array}
\right .
;\hspace{0.5 cm}
m(x)=\left \{
\begin{array}{l}
(\beta + 1)x\hspace{0.3 cm}\text{ if }x\geq 0\\
(1-\beta)x\hspace{0.3 cm}\text{ if }x< 0.
\end{array}
\right .
$$
(see \cite{lejay-2006}).
In particular, the density $\ell^{\beta,0}(t,x,y)dy$ of the SBM w.r.t. the speed measure $m(dy)$ is given by
$$\ell^{\beta,0}(t,x,y)=
\left \{
\begin{array}{lll}
&\displaystyle\frac{1}{(1+\beta)\sqrt{2\pi t}}\big(  \exp\{ -\frac{(y-x)^2}{2t} \}  - \exp\{ -\frac{(y+x)^2}{2t} \}  \big)\\
&\displaystyle+\frac{1}{\sqrt{2\pi t}}\exp\big\{ -\frac{(x+y)^2}{2t} \big\},\text{ if }x>0,y>0 ;\\
&\displaystyle\frac{1}{\sqrt{2\pi t}}\exp\big\{ -\frac{(x - y)^2}{2t} \big\},\text{ if }x>0,y<0\text{ or if }x<0,y>0 ; \\
&\displaystyle\frac{1}{(1-\beta)\sqrt{2\pi t}}\big(  \exp\{ -\frac{(y-x)^2}{2t} \}  - \exp\{ -\frac{(y+x)^2}{2t} \}  \big)\\
&\displaystyle+\frac{1}{\sqrt{2\pi t}}\exp\big\{ -\frac{(x + y)^2}{2t}\big \},\text{ if }x<0,y<0.
\end{array}
\right .
$$
(of course $\ell^{\beta,0}(t,x,y) =\ell^{\beta,0}(t,y,x)$).

Let
$$
M^{\beta,0}_T:=\sup_{0\leq s\leq t}B^{\beta, 0}_s~;~\hspace{0.3 cm}\rho^{\beta,0}_T:=\inf\{s\geq 0~:~B^{\beta,0}_s=M^{\beta,0}_s\}.
$$
Then, applying the results of Theorem 2 in Pitman-Yor  \cite{Pitman-Yor}, we have the following proposition~:
\begin{prop}
\hfill
\begin{enumerate}
\item 
For any $a,b\leq z<\infty$, $\lambda > 0$, we have that
\begin{equation}
{\mathbb E}\pare{{\rm e}^{-\lambda\rho^{\beta,0}_T}\id{M^{\beta,0}_T\in dz}\,\,~\,|\,~B^{\beta, 0}_0 = a, B^{\beta, 0}_T = b} = 
\frac{u_\lambda(a;z)u_\lambda(z;b)}{\ell^{\beta,0}(T,a,b)}s(dz).
\end{equation}
\item Moreover, under ${\mathbb P}\pare{.\,\,~|~\,B^{\beta, 0}_0 = a, B^{\beta, 0}_T = b, M^{\beta,0}_T = z, \rho^{\beta,0}_T = u}$, the path fragments
$$
\pare{B^{\beta, 0}_s~:~0\leq s\leq u}\hspace{0.4 cm}\pare{B^{\beta, 0}_{T-s}~:~0\leq s\leq T-u}
$$
are independent, distributed respectively like
$$
\pare{B^{\beta, 0}_s~:~0\leq s\leq \tau_z^{\beta, 0}}\hspace{0.2 cm}\text{under}\hspace{0.2 cm}{\mathbb P}\pare{.\,\,~|~\,B^{\beta, 0}_0 = a}\hspace{0.2 cm}\text{given}\hspace{0.2 cm}\tau_z^{\beta, 0} = u
$$
and
$$
\pare{B^{\beta, 0}_s~:~0\leq s\leq \tau_z^{\beta, 0}}\hspace{0.2 cm}\text{under}\hspace{0.2 cm}{\mathbb P}\pare{.\,\,~|~\,B^{\beta, 0}_0 = b}\hspace{0.2 cm}\text{given}\hspace{0.2 cm}\tau_z^{\beta, 0} = T-u.
$$
\end{enumerate}
\end{prop}
An open problem is to find a description of these laws and to give a procedure in order to simulate these laws exactly.

		\subsection{Concluding remarks}
In this paper we presented an extension of the exact simulation method of \cite{beskos2} that permits to sample an exact skeleton
of a one dimensional diffusion process skewed at $0$. This method may be applied to diffusions related to strongly elliptic divergence form operators that possess
a discontinuous coefficient at $0$. The basic idea of this contribution depends highly  on the possibility to perform a Girsanov transformation such that no local time appears in the Girsanov exponential weight and such that the reference measure is tractable.

In our opinion, this first work should be extended in several directions.

Firstly, it is necessary to give a complete treatment of the case $\beta=0$ (see Remark \ref{rem-betaneqzero}). 
In this case, there still exists a way to perform a Girsanov transformation such that no local time appears in the Girsanov exponential weight, but then  the reference measure becomes that of a Brownian motion with two-valued drift  (see \cite{kara} for an introduction to these particular types of Brownian motions). As before, the difficulty arises for the simulation of the bridges in the Step 3 of the algorithm.

Further digging shows that, in this particular situation, the solution of the exact simulation problem in the manner of \cite{beskos2} is closely related to the computation of joint laws for the position together with local and occupation times by an arbitrary Brownian bridge (with no drift) but conditioned on its final position and local time at $0$. 
Even if there exists abundant litterature dealing with the Brownian bridge, there is no result for such joint laws. 

Consequently, we believe that performing a totally exact simulation algorithm in full generality in the case $\beta=0$ appears to be outside the scope of this paper. Note that a satisfactory treatment of the case $\beta=0$ is crucial if one has the objective to deal with the even more general case, where the discontinuity of $\overline{b}$ and the local time appear at distinct space points. 

\medskip

Secondly, various questions arise in the treatment of "skewness"~:~how can we overcome the restraining boundedness assumption on the drift function $\bar{b}$ ? What about a one dimensional diffusion process skewed at a finite number of points ?

\section*{Acknowledgements}
We would like to thank the anonymous referee for helpful comments and having pointed the reference \cite{Zaitseva-2}.

\bibliographystyle{plain}
\bibliography{simple_sim_exact.bib}
\end{document}